\documentclass[a4paper,12pt, twoside, reqno]{amsart}
\usepackage{amssymb}
\usepackage{amsmath}
\newtheorem{theorem}{Theorem}
\newtheorem{corollary}[theorem]{Corollary}
\newtheorem{lem}{Lemma}

\newtheorem{defn}{Definition}
\newtheorem{ex}{Example}
\newtheorem{remark}{Remark}[section]

\title{Weak and strong convergence theorems for enriched strictly pseudocontractive operators in Hilbert spaces}
\author{Vasile BERINDE}
\begin{document}
\maketitle \pagestyle{myheadings} \markboth{Vasile Berinde} {Enriched strictly pseudocontractive operators...}
\begin{abstract}
In this paper, we introduce and study the class of {\it enriched strictly pseudocontractive mappings}  in Hilbert spaces and  extend the corresponding convergence theorem (Theorem 12) in [Browder, F. E., Petryshyn, W. V., {\it Construction of fixed points of nonlinear mappings in Hilbert space}, J. Math. Anal. Appl. {\bf 20} (1967), 197--228] and Theorem 3.1 in [Marino, G., Xu, H.-K., {\it Weak and strong convergence theorems for strict pseudo-contractions in Hilbert spaces}, J. Math. Anal. Appl. {\bf 329} (2007), no. 1, 336--346], from the class of strictly pseudocontractive mappings to that of enriched strictly pseudocontractive mappings and thus include many other important related results from literature as particular cases.
\end{abstract}

\section{Introduction}

Let $C$ be a nonempty subset of a normed space $X$ and $T:C\rightarrow C$ a mapping. 

\indent (i) $T$ is called {\it nonexpansive} if
$$
\|Tx-Ty\|\leq \|x-y\|, \forall x,y\in C.
$$

(ii)  $T$ is called {\it strictly pseudocontractive} if there exists $k<1$ such that
$$
\|Tx-Ty\|^2\leq \|x-y\|^2+k\|x-y-(Tx-Ty)\|,\forall x,y \in C.
$$
$T$ is also called $k${\it -strictly pseudocontractive}. 

If $k=1$ in the previous inequality, $T$ is called {\it pseudocontractive}.

It is easy to see that every nonexpansive mapping is strictly pseudocontractive and every strictly pseudocontractive mapping is pseudocontractive but the reverses are not more true. Moreover, like in the case of nonexpansive mappings, any strictly pseudocontractive operator is Lipschitz continuous (\cite{Os}), i.e.,
$$
\|Tx-Ty\|\leq L \|x-y\|,\forall x,y \in C,\, (L>0).
$$ 
An element $x\in C$ is said to be a {\it fixed point} of $T$ is $Tx=x$. Denote by $Fix\,(T)$ the set of all fixed points of $T$ and recall the following concepts.

The notion of strictly pseudocontractive operator in Hilbert spaces has been introduced and studied by Browder and Petryshyn \cite{BroP67}, where the following result has been established.
\begin{theorem}  \label{th1}
Let $C$ be a bounded closed convex
subset of a Hilbert space $H$ and  $T:C\rightarrow C$ be a $k$-strictly pseudocontractive  mapping. Then, for any given $x_0\in C$ and any fixed number $\gamma$ such that $1-k<\gamma<1$, the sequence $\{x_n\}_{n=0}^{\infty}$ given by
\begin{equation} \label{eq4}
x_{n+1}=(1-\gamma) x_n+\gamma Tx_n,\,n\geq 0,
\end{equation}
converges weakly to some fixed point of $T$. If, in addition, $T$ is demicompact, then $\{x_n\}_{n=0}^{\infty}$ converges strongly to some fixed point of $T$.
\end{theorem}
For some early developments emerging from Browder and Petryshyn paper \cite{BroP67}, we refer to \cite{Bet},  \cite{Deng}, \cite{Goh}, \cite{Joh}, \cite{Kha}, \cite{Mar}, \cite{Mol}, \cite{Muk}, \cite{Rho}, \cite{Os}, \cite{Schu91}, \cite{Schu91a}, \cite{Weng}, \cite{Zhou} etc.

On the other hand, the author in \cite{Ber19a}, generalized the notion of nonexpansive mapping by introducing the concept of $b$-enriched nonexpansive mappings, as follows. 

Let $(X,\|\cdot\|)$ be a linear normed space. A mapping $T:X\rightarrow X$ is said to be  {\it enriched nonexpansive} if there exists $b\in[0,ü\infty)$ such that
\begin{equation} \label{eq3a}
\|b(x-y)+Tx-Ty\|\leq (b+1) \|x-y\|,\forall x,y \in X.
\end{equation}
Note that nonexpansive mappings correspond to the case $b=0$ in \eqref{eq3a}.

In \cite{Ber19a} it has been shown that, under appropriate circumstances, any $b$-enriched nonexpansive mapping has at least one fixed point and then it was presented an iterative algorithm for the computation of the fixed points of $b$-enriched nonexpansive mappings for which there were presented weak and strong convergence theorems.

Starting from these facts, the aim of this paper is twofold: first, to introduce and study the class of enriched strictly pseudocontractive mappings, which naturally include the  enriched nonexpansive mappings as well as the strictly pseudocontractive mappings in the sense of Browder and Petryshyn \cite{BroP67}, and, secondly, to present a constructive method for approximating the fixed points of $b$-enriched strictly pseudocontractive mappings in Hilbert spaces.

We shall need in the proof of our main results the following two definitions.
\begin{defn} \cite{Pet} \label{def1}
Let $H$ be a Hilbert space and $C$ a
subset of $H$. A mapping $T:C\rightarrow H$ is called \textit{demicompact} 
if it has the property that whenever $\{u_{n}\}$ is a bounded sequence in $H$
and $\{Tu_{n}-u_{n}\}$ is strongly convergent, then there exists a
subsequence $\{u_{n_{k}}\}$ of $\{u_{n}\}$ which is strongly convergent.
\end{defn}

\begin{defn} \cite{BroP67} \label{def2}
Let $H$ be a Hilbert space and $C$ a closed convex
subset of $H$. A mapping $T:C\rightarrow C$ is called \textit{asymptotically regular} (on $C$)
if, for each $x\in C$, 
$$
\|T^{n+1}x-T^n x\|\rightarrow 0  \textnormal{ as } n\rightarrow \infty.
$$ 
\end{defn}

The following Lemma, which is adapted after Corollary to Theorem 5 in \cite{BroP67} will be also used in the proof of the main result of the next section.
\begin{lem} \label{lem1}
Let $H$ be a Hilbert space and $C$ a closed convex
subset of $H$. If the mapping $U:C\rightarrow C$ is nonexpansive and $Fix\,(U)\neq \emptyset$ then, for any given $\lambda\in (0,1)$, the mapping $U_\lambda=I+(1-\lambda) U$ maps $C$ into $C$, has the same fixed points as $U$ and is asymptotically regular.
\end{lem}

\section{Approximating fixed points of enriched strictly pseudocontractive mappings by Krasnoselskij iteration} 

\begin{defn}\label{def0}
Let $(X,\|\cdot\|)$ be a linear normed space. A mapping $T:X\rightarrow X$ is said to be an {\it enriched strictly pseudocontractive mapping} if there exist $b\in[0,ü\infty)$ and $k<1$ such that for all $x,y \in X$,
\begin{equation} \label{eq3}
\|b(x-y)+Tx-Ty\|^2\leq (b+1)^2 \|x-y\|^2+k\|x-y-(Tx-Ty)\|^2.
\end{equation}
We shall also call $T$ as a  $(b,k)$-{\it enriched strictly pseudocontractive mapping}. 
\end{defn}

\begin{ex}[ ] \label{ex1}
1) Any strictly pseudocontractive mapping $T$ is a $(0,k)$-enriched strictly pseudocontractive mapping, i.e., it satisfies \eqref{eq3} with $b=0$.  

2) Any  $b$-enriched nonexpansive mapping is a $(b,k)$-enriched strictly pseudocontractive mapping for any $k<1$. 
\end{ex}

Now we can state and prove the main result of this section. 
\begin{theorem}  \label{th1}
Let $C$ be a bounded closed convex
subset of a Hilbert space $H$ and  $T:C\rightarrow C$ be a $(b,k)$-enriched strictly pseudocontractive and demicompact mapping. Then  $Fix\,(T)\neq \emptyset$ and, for any given $x_0\in C$ and any fixed number $\gamma$, such that $0<\gamma<1-k$, the Krasnoselskij iteration $\{x_n\}_{n=0}^{\infty}$ given by
\begin{equation} \label{eq4}
x_{n+1}=(1-\gamma) x_n+\gamma Tx_n,\,n\geq 0,
\end{equation}
converges strongly  to a fixed point of $T$. 
\end{theorem}

\begin{proof}
Using the fact that $T$ is $(b,k)$-enriched strictly pseudocontractive and denoting $b=\dfrac{1}{\lambda}-1$, it follows that
 $\lambda\in (0,1]$ and thus \eqref{eq3} is equivalent to
\begin{equation} \label{eq5}
\|(1-\lambda)(x-y)+\lambda Tx-\lambda Ty\|^2\leq  \|x-y\|^2+k\|\lambda x-\lambda y-\lambda(Tx-Ty)\|^2,
\end{equation}
 for all $x,y \in C$. 
 
 Denote $T_\lambda x=(1-\lambda) x+\lambda T x$, where $I$ is the identity map. Then inequality \eqref{eq5} expresses the fact that 
\begin{equation} \label{eq7-1}
\|T_\lambda x-T_\lambda y\|^2\leq  \|x-y\|^2+k\|x-y-(T_\lambda x-T_\lambda y)\|^2,\forall x,y \in C,
\end{equation}
i.e., the averaged operator $T_\lambda$ is $k$-strictly pseudocontractive. 

Next, we prove that the operator $U=I-T_\lambda$ satisfies the inequality
\begin{equation} \label{eq6}
\langle Ux-Uy,x-y\rangle \geq \frac{1-k}{2}\cdot \|Ux-Uy\|^2,\, \forall x,y \in C,
\end{equation}
where $k$ is the constant involved in \eqref{eq3}. Indeed, by \eqref{eq7-1} we have
\begin{equation} \label{eq8-1}
\|T_\lambda x-T_\lambda y\|^2\leq  \|x-y\|^2+k\|U x-U y)\|^2,\forall x,y \in C.
\end{equation}
On the other hand, 
$$
\|T_\lambda x-T_\lambda y\|^2=\|(I-U)x-(I-U)y\|^2=\|x-y-(Ux-Uy)\|^2
$$
$$
=\|x-y\|^2+\|Ux-Uy\|^2-2\cdot \langle Ux-Uy,x-y\rangle
$$
which, by \eqref{eq8-1} implies
$$
\|x-y\|^2+\|Ux-Uy\|^2-2\cdot \langle Ux-Uy,x-y\rangle\leq \|x-y\|^2+k\|U x-U y)\|^2
$$
from which we get the desired inequality \eqref{eq6}.

For $t>0$, consider  the map $U_t=(1-t) I+t T_\lambda$, where $I$ is the identity map,
which means $U_t=(1-t)I+t(I-U)=I-t\cdot U$. Hence
$$
\|U_t(x)-U_t(y)\|^2=\|x-y-t(Ux-Uy)\|^2=\|x-y\|^2+t^2\|Ux-Uy\|^2
$$
$$
-2t \langle Ux-Uy,x-y\rangle,
$$
which, by means of \eqref{eq6} yields
$$
\|U_t(x)-U_t(y)\|^2\leq \|x-y\|^2 +(t^2-(1-k)t)\|Ux-Uy\|^2.
$$
This shows that, for any $t$ satisfying $0<t<1-k$, the mapping
$$
U_t=I-t\cdot U=(1-t) I+t\cdot T_\lambda=(1-\lambda t) I+\lambda t T
$$
has the property
$$
\|U_t(x)-U_t(y)\|\leq  \|x-y\|,\,\forall x,y\in C,
$$
that is, $U_t$ is nonexpansive. 
In order to prove the last part of the theorem, consider the sequence $\{x_n\}_{n=0}^{\infty}$ given by 
$$
x_{n+1}=(1-\mu) x_n+\mu U_t x_n,\,n\geq 0,
$$
where $0<\mu <1$. It is obvious that $\{x_n\}_{n=0}^{\infty}$ lies in $C$ and hence it is bounded.
Denote $V_\mu=(1-\mu) I+\mu U_t$. Then, since $U_t$ is nonexpansive, by Lemma \ref{lem1} it follows that $V_\mu$ is asymptotically regular, i.e.,
\begin{equation} \label{eq8b}
\|x_n-V_\mu x_{n}\|\rightarrow 0\, \textnormal{ as } n\rightarrow \infty.
\end{equation}
We also have
\begin{equation} \label{eq8a}
V_\mu x-x=\mu (U_t x-x)=t \lambda \mu (Tx-x),
\end{equation}
and hence
$$
\|x_n-U_t x_{n}\|\rightarrow 0,\, \textnormal{ as } n\rightarrow \infty.
$$
Since, by hypothesis, $T$ is demicompact, it follows by \eqref{eq8a} that $V_\mu$
 is demicompact, too. Hence, in view of \eqref{eq8b}, there exists a subsequence $\{x_{n_k}\}$ of $\{x_n\}_{n=0}^{\infty}$ which converges strongly in $C$.
 Denote
 $$
 \lim_{k\rightarrow \infty} x_{n_k}=q.
 $$
 Then, by the continuity of $U_t$ it follows that $V_\mu$ is continuous and hence
 $$
 V_\mu x_{n_k} \rightarrow V_\mu q,\, \textnormal{ as } k\rightarrow \infty.
 $$
 Therefore, $\{x_{n_k} -V_\mu x_{n_k} \}$ converges strongly to $0$ and simultaneously, $\{x_{n_k} -V_\mu x_{n_k} \}$ converges strongly to $q-V_\mu q$, which proves that $q=V_\mu q$, i.e., $$q\in Fix\,(V_\mu)=Fix\,(U_t)=Fix\,(T_\lambda)=Fix\,(T).$$

The convergence  of the entire sequence $\{x_n\}_{n=0}^{\infty}$ to $q$ now follows from the inequality
$$
\|x_{n+1}-q\|\leq \|x_{n}-q\|, \,n\geq 0,
$$
which is a direct consequence of the nonexpansivity of $V_\mu$ (which, in turn, is a consequence of the nonexpansivity of $U_t$). 

Hence, for any $x_0\in C$, the Krasnoselskij iteration $\{x_n\}_{n=0}^{\infty}$, given by
$$
x_{n+1}=V_\mu x_n=(1-\mu) x_n+\mu U_t x_n=(1-\mu) x_n+\mu [(1-\lambda t)x_n+\lambda t T x_n]
$$
$$
=(1-\lambda \mu t) x_n+\lambda \mu t Tx_n
$$
converges strongly to $q\in Fix\,(T)$ as $n\rightarrow \infty$.

To get exactly the formula \eqref{eq4} for the Krasnoselskij iteration $\{x_n\}_{n=0}^{\infty}$ given above, we simply denote $\gamma=\lambda \mu t$ and use the fact that $t\in (0,1-k)$ to get the inequality $\gamma\in (0,1-k)$.
\end{proof}

\begin{remark}\label{rem1}
Theorem \ref{th1} is an extension of Theorem 12  in Browder and Petryshyn \cite{BroP67}, by considering instead of strictly pseudocontractive mappings the larger class of enriched strictly pseudocontractive  mappings.
\end{remark}
As a corollary of Theorem \ref{th1}, we present the strong convergence part of Theorem 12 in \cite{BroP67}.
\begin{corollary} (Theorem 12, \cite{BroP67}) \label{cor1}
Let $C$ be a bounded closed convex
subset of a Hilbert space $H$ and  $T:C\rightarrow C$ be a  $k$-strictly pseudocontractive and demicompact operator. Then $Fix\,(T)\neq \emptyset$  and for any given $x_0\in C$ and any fixed number $\gamma$, $0<\gamma<1-k$, the Krasnoselskij iteration $\{x_n\}_{n=0}^{\infty}$ given by
$$
x_{n+1}=(1-\gamma) x_n+\gamma Tx_n,\,n\geq 0,
$$
converges strongly to a fixed point of $T$.
\end{corollary}

\begin{proof}
A $k$-strictly pseudocontractive mapping is a $(0,k)$-enriched $k$-strictly pseudocontractive mapping. Hence, Corollary \ref{cor1} follows from Theorem \ref{th1} for $b=0$ , that is, for $\lambda=1$.
\end{proof}

We now prove a weak convergence theorem which is an extension of the weak convergence part of Theorem 12 in \cite{BroP67} and corresponds to the case when $T$ is not more demicompact.

\begin{theorem}  \label{th4}
Let $C$ be a bounded closed convex
subset of a Hilbert space $H$ and  $T:C\rightarrow C$ be a $(b,k)$-enriched strictly pseudocontractive mapping. Then  $Fix\,(T)\neq \emptyset$ and, for any given $x_0\in C$ and any fixed number $\gamma$, such that $0<\gamma<1-k$, the Krasnoselskij iteration $\{x_n\}_{n=0}^{\infty}$ given by
$$
x_{n+1}=(1-\gamma) x_n+\gamma Tx_n,\,n\geq 0,
$$
converges weakly  to a fixed point of $T$. 
\end{theorem}

\begin{proof}
We proceed like in the first part of the proof of Theorem \ref{th1} and get the conclusion that the mapping $U_t=(1-\lambda t)I+\lambda t T$ is nonexpansive.

To prove the theorem, it suffices to show that if  $\{x_{n_j}\}_{j=0}^{\infty}$ given by 
$$
x_{n_j+1}=(1-\mu) x_{n_j}+\mu U_t x_{n_j},\,j\geq 0,
$$
converges weakly to a certain $p_0$, then $p_0$ is a fixed point of $U_t$ (and of $T_\lambda=(1-\lambda) I+\lambda T$) and hence of $T$ and therefore $p_0=p$. 

Suppose that $\{x_{n_{j}}\}_{j=0}^{\infty}$  does not converge weakly to $p$. 

Using the same arguments like in the proof of Theorem \ref{th1}, we obtain that $V_\mu=(1-\mu) I+\mu U_t$ is nonexpansive and asymptotically regular, that is,
$$
\|x_{n_j}-V_\mu x_{n_j}\|\rightarrow 0,\, \textnormal{ as } j\rightarrow \infty.
$$
On the other hand
$$
\|x_{n_j}-V_\mu p_0\|\leq \|V_\mu x_{n_j}-V_\mu p_0\|+\|x_{n_j}-V_\mu x_{n_j}\|
$$
$$
\leq \|x_{n_j}-p_0\|+\|x_{n_j}-V_\mu x_{n_j}\|,
$$
which implies that
\begin{equation} \label{eq7}
\limsup \left(\|x_{n_j}-V_\mu p_0\|-\|x_{n_j}-p_0\| \right)\leq 0.
\end{equation}
Similarly to the proof of Theorem \ref{th1} in \cite{Ber19a}, we have
$$
\|x_{n_j}-V_\mu p_0\|^2=\|(x_{n_j}-p_0)+(p_0-V_\mu p_0\|^2
$$
$$
=\|x_{n_j}-p_0\|^2+\|p_0-V_\mu p_0\|^2+2\langle x_{n_j}-p_0,p_0-V_\mu p_0\rangle,
$$
which, together with the fact that $\{x_{n_j}\}$ converges weakly to $p_0$, implies
\begin{equation} \label{eq8}
\lim_{j\rightarrow \infty} \left[\|x_{n_j}-V_\mu p_0\|^2-\|x_{n_j}-p_0\|^2\right]=\|p_0-V_\mu p_0\|^2.
\end{equation}
We also have
$$
\|x_{n_j}-V_\mu p_0\|^2-\|x_{n_j}-p_0\|^2=\left(\|x_{n_j}-V_\mu p_0\|-\|x_{n_j}-p_0\|\right)\cdot
$$
\begin{equation} \label{eq9}
\left(\|x_{n_j}-V_\mu p_0\|+\|x_{n_j}-p_0\|\right).
\end{equation}
Since $C$ is bounded, the sequence $\{\|x_{n_j}-V_\mu p_0\|+\|x_{n_j}-p_0\|\}$ is bounded, too, and therefore by combining \eqref{eq7}, \eqref{eq8} and \eqref{eq9}, we get
$$
\|p_0-V_\mu p_0\|=0,
$$
that is,
$$
V_\mu p_0=p_0,
$$
which implies
$$
 p_0\in Fix\,(V_\mu)=Fix\,(U_t)=Fix\,(T)=\{p\}.
$$
\end{proof}

\section{Approximating fixed points of enriched strictly pseudocontractive mappings by Mann iteration} 

The aim of this section is to show that, under appropriate assumptions, we can obtain convergence theorems for Mann type iterative scheme, which generates a sequence $\{x_n\}$ in the following manner:
\begin{equation} \label{3.1}
x_{n+1}=(1-\gamma_n) x_n+\gamma_n Tx_n,\,n\geq 0,
\end{equation}
where $\gamma_n$ is a real sequence in $(0,1)$ and $x_0$ is the initial guess. Note that Krasnoselskij iteration \eqref{eq4} is obtained from \eqref{3.1} when the sequence $\{\gamma_n\}$ is constant.

Marino and Xu \cite{Mar} have proved weak and strong convergence theorems, by using algorithm \eqref{3.1} in order to approximate fixed points of $k$-strictly pseudo-contractive mappings, with the control sequence $\gamma_n$ satisfying some appropriate conditions. The aim of this section is to extend Marino and Xu \cite{Mar} and other related results from $k$-strictly pseudo-contractive mappings to enriched $(b,k)$-strictly pseudo-contractive mappings.

In order to establish our convergence theorems, we collect some auxiliary results mainly taken from Marino and Xu \cite{Mar}.

\begin{lem} {\rm (Lemma 1.1, \cite{Mar})}\label{lem2}
 Let $H$ be a real Hilbert space. Then the following identities hold:
$$
(i) \quad \|x\pm y\|^2=\|x\|^2\pm 2\langle x,y\rangle+\|y\|^2,\forall x,y\in H;
$$
$$
(ii) \quad \|tx+(1-t)y\|^2=t\|x\|^2+(1-t)\|y\|^2-t(1-t)\|x-y\|^2,\forall t\in [0,1],\forall x,y\in H.
$$
(iii) If $\{x_n\}$ is a sequence in $H$ weakly convergent to $z$, then
$$
\limsup_{n\rightarrow \infty} \|x_n-y\|^2=\limsup_{n\rightarrow \infty} \|x_n-z\|^2+\|z-y\|^2,\forall y\in H.
$$
\end{lem}

\begin{lem} {\rm (Proposition 2.1, (ii),\cite{Mar})}\label{lem3}
Let $C$ be a closed convex subset of a Hilbert space $H$ and $T:C\rightarrow C$ be a $k$-strict pseudo contraction. Then $I-T$ is demiclosed (at $0$), that is, if $\{x_n\}$ is a sequence in $C$ such that $x_n\rightharpoonup \overline{x}$ and $(I-T)x_n\rightarrow 0,$ then $(I-T)\overline{x}=0$.
\end{lem}

\begin{theorem}  \label{th5}
Let $C$ be a bounded closed convex
subset of a Hilbert space $H$ and  $T:C\rightarrow C$ be a $(b,k)$-enriched strictly pseudocontractive mapping for some $0\leq k<1$. Then  $Fix\,(T)\neq \emptyset$ and, for any given $x_0\in C$ and any control sequence $\{\alpha_n\}$, such that $k<\alpha_n<1$ and 
\begin{equation} \label{eq7-1u}
\sum_{n=0}^{\infty} (\alpha_n-k)(1-\alpha_n) =\infty,
\end{equation}
the Krasnoselskij-Mann iteration $\{x_n\}_{n=0}^{\infty}$ given by
$$
x_{n+1}=(1-\lambda\alpha_n) x_n+\lambda \alpha_n T x_n],\,n\geq 0,
$$
for some $\lambda\in (0,1)$, converges weakly  to a fixed point of $T$. 
\end{theorem}

\begin{proof}
Using the fact that $T$ is $(b,k)$-enriched strictly pseudocontractive and denoting $b=\dfrac{1}{\lambda}-1$, it follows that
 $\lambda\in (0,1]$ and thus from \eqref{eq3} we obtain
\begin{equation} \label{eq5a}
\|(1-\lambda)(x-y)+\lambda Tx-\lambda Ty\|^2\leq  \|x-y\|^2+k\|\lambda x-\lambda y-\lambda(Tx-Ty)\|^2,
\end{equation}
 for all $x,y \in C$. 
 
 Denote $T_\lambda x=(1-\lambda) x+\lambda T x$, where $I$ is the identity map. Then inequality \eqref{eq5a} expresses the fact that 
\begin{equation} \label{eq7-1a}
\|T_\lambda x-T_\lambda y\|^2\leq  \|x-y\|^2+k\|x-y-(T_\lambda x-T_\lambda y)\|^2,\forall x,y \in C,
\end{equation}
i.e., the averaged operator $T_\lambda$ is $k$-strictly pseudocontractive. By Theorem \ref{th1}, $Fix\,(T_\lambda)=Fix\,(T)\neq \emptyset$. Let $p\in Fix\,(T_\lambda)$.

By Lemma \ref{lem2} (ii), we have
$$
\|x_{n+1}-p\|^2=\|\alpha_n(x_n-p)+(1-\alpha_n)(T_\lambda x_n-p)\|^2=\alpha_n \|x_n-p\|^2
$$
$$
+(1-\alpha_n)\|T_\lambda x_n-p\|^2-\alpha_n(1-\alpha_n)\|x_n-T_\lambda x_n\|^2\leq \alpha_n \|x_n-p\|^2
$$
$$
+(1-\alpha_n)(\|T_\lambda x_n-p\|^2+k\|x_n-T_\lambda x_n\|^2)-\alpha_n(1-\alpha_n)\|x_n-T_\lambda x_n\|^2
$$
\begin{equation} \label{eq7-1b}
=\|x_n-p\|^2-(\alpha_n-k)(1-\alpha_n)\|x_n-T_\lambda x_n\|^2. 
\end{equation}
Based on the properties of $\{\alpha_n\}$, we get 
$$
\|x_{n+1}-p\|\leq \|x_{n}-p\|
$$
which shows that the sequence $\{\|x_{n}-p\|\}$ is decreasing. 

Hence $\lim\limits_{n\rightarrow \infty} \|x_{n}-p\|$ exists. By \eqref{eq7-1b} we also get
\begin{equation} \label{eq7-1c}
\sum_{n=0}^{\infty} (\alpha_n-k)(1-\alpha_n)\|x_n-T_\lambda x_n\|^2\leq \|x_0-p\|^2<\infty.
\end{equation}
Using condition \eqref{eq7-1a},  by \eqref{eq7-1c} we conclude that
\begin{equation} \label{eq7-1d}
\liminf_{n\rightarrow \infty} \|x_n-T_{\lambda} x_n\|=0.
\end{equation}
We now prove that actually $
\liminf\limits_{n\rightarrow \infty} \|x_n-T_{\lambda} x_n\|
$ exists. To this end, we prove that $\{\|x_n-T_{\lambda} x_{n}\|\}$ is decreasing. Indeed, since $x_n-x_{n+1}=(1-\alpha_n)(x_n-T_{\lambda} x_n)$, we have
$$
\|x_n-T_\lambda x_n\|^2=\|\alpha_n(x_n-T_{\lambda} x_{n+1})+(1-\alpha_n)(T_\lambda x_n-T_\lambda x_{n+1})\|^2
$$
$$
=\alpha_n \|x_n-T_{\lambda} x_{n+1}\|^2+(1-\alpha_n)\|T_\lambda x_n-T_\lambda x_{n+1})\|^2-\alpha_n (1-\alpha_n)\|x_n-T_{\lambda} x_{n}\|^2
$$
$$
\leq \alpha_n \|(x_n-x_{n+1})+(x_{n+1}-T_{\lambda} x_{n+1})\|^2-\alpha_n (1-\alpha_n)\|x_n-T_{\lambda} x_{n}\|^2
$$
$$
+(1-\alpha_n)\left[ \|x_n-x_{n+1}\|^2+k\|(x_n-T_{\lambda} x_{n})-(x_{n+1}-T_{\lambda} x_{n+1})\|^2\right]
$$
$$
=\alpha_n (\|x_n-x_{n+1}\|^2+\|x_{n+1}-T_{\lambda} x_{n+1}\|^2)+2\langle x_n-x_{n+1},x_{n+1}-T_{\lambda} x_{n+1}\rangle
$$
$$
-\alpha_n (1-\alpha_n)\|x_n-T_{\lambda} x_{n}\|^2+(1-\alpha_n)\left[\|x_n-x_{n+1}\|^2+k\|x_n-T_{\lambda} x_{n}\|^2\right.
$$
$$
\left.+k(\|x_{n+1}-T_{\lambda} x_{n+1}\|^2-2\langle x_n-x_{n+1},x_{n+1}-T_{\lambda} x_{n+1}\rangle)\right]
$$
$$
=(1-\alpha_n)^2 \|x_n-T_{\lambda} x_{n}\|^2+\alpha_n \|x_{n+1}-T_{\lambda} x_{n+1}\|^2
$$
$$
+2\alpha_n(1-\alpha_n)\langle x_n-T_{\lambda}x_{n},x_{n+1}-T_{\lambda} x_{n+1}\rangle-\alpha_n (1-\alpha_n)\|x_n-T_{\lambda} x_{n}\|^2
$$
$$
+k(1-\alpha_n)(\|x_n-T_{\lambda} x_{n}\|^2+\|x_{n+1}-T_{\lambda} x_{n+1}\|^2-2\langle x_n-T_{\lambda} x_{n},x_{n+1}-T_{\lambda} x_{n+1}\rangle)
$$
$$
=\left[\alpha_n+k(1-\alpha_n)\right]\|x_{n+1}-T_{\lambda} x_{n+1}\|^2+(1-\alpha_n)(1+k-2\alpha_n)\|x_n-T_{\lambda} x_{n}\|^2
$$
$$
+2(\alpha_n-k)(1-\alpha_n)\langle x_n-T_{\lambda} x_{n},x_{n+1}-T_{\lambda} x_{n+1}\rangle
$$
$$
\leq \left[\alpha_n+k(1-\alpha_n)\right]\|x_{n+1}-T_{\lambda} x_{n+1}\|^2+(1-\alpha_n)(1+k-2\alpha_n)\|x_n-T_{\lambda} x_{n}\|^2
$$
$$
+2(\alpha_n-k)(1-\alpha_n)\|x_n-T_{\lambda} x_{n}\|\|x_{n+1}-T_{\lambda} x_{n+1}\|.
$$
Denote $\beta_n=\|x_n-T_{\lambda} x_{n}\|$, for each $n$ and we get
$$
(1-\alpha_n)(1-k)\beta_{n+1}^2\leq (1-\alpha_n)(1+k-2\alpha_n)\beta_n^2+2(1-\alpha_n)(\alpha_n-k)\beta_n\beta_{n+1}.
$$
We may assume $\beta_n>0$. Since $1-\alpha_n>0$, by denoting $\delta_n=\dfrac{\beta_{n+1}}{\beta_{n}}$, by the previous inequality we obtain the following quadratic inequality
$$
(1-k)\delta_n^2-2(\alpha_n-k)\delta_n-(1+k-2\alpha_n)\leq 0
$$
\begin{equation} \label{f1}
\Longleftrightarrow (\delta_n-1)[(1-k)\delta_n+1+k-2\alpha_n]\leq 0.
\end{equation}
Now, since, by hypothesis, 
$$(1-k)\delta_n+1+k-2\alpha_n=(1-k)\delta_n+1-\alpha_n+k-\alpha_n>0,$$
from \eqref{f1} we obtain $\delta_n\leq 1$, that is, $\beta_{n+1}\leq \beta_n$, which shows that, indeed, $
\lim\limits_{n\rightarrow \infty} \|x_n-T_{\lambda} x_n\|
$
exists. By \eqref{eq7-1d}, this means that
\begin{equation} \label{eq7-1e}
\lim_{n\rightarrow \infty} \|x_n-T_{\lambda} x_n\|=0.
\end{equation}
Denote, as usually, the weak $\omega$-limit of a sequence $\{x_n\}$, $\{\overline{x}: \exists x_{n_j}\rightharpoonup \overline{x}\}$, by $\omega_w(x_n)$. 

Since $T_{\lambda}$ is strictly pseudo-contractive, by Lemma \ref{lem3}  we deduce that $I-T_{\lambda}$ is demiclosed (at $0$), that is,  $\omega_w(x_n)\subset Fix\,(T_{\lambda})$. 

To prove that $\{x_n\}$ is actually weakly convergent, we take $p,q\in \omega_w(x_n)$ and let $\{x_{n_i}\}$ and $\{x_{m_j}\}$ be subsequences of $\{x_n\}$ such that $x_{n_i}\rightharpoonup p$ and $x_{m_j}\rightharpoonup q$, respectively.

Since $\lim\limits_{n\rightarrow \infty} \|x_n-z\|$ exists for every $z \in Fix\,(T_{\lambda})$ and since $p,q\in Fix\,(T_{\lambda})$, by Lemma \ref{lem2} (iii), we obtain
$$
\lim_{n\rightarrow \infty} \|x_n-p\|^2=\lim_{j\rightarrow \infty} \|x_{m_j}-p\|^2=\lim_{j\rightarrow \infty} \|x_{m_j}-q\|^2+\|q-p\|^2
$$
$$
=\lim_{i\rightarrow \infty} \|x_{n_i}-q\|^2+\|q-p\|^2=\lim_{i\rightarrow \infty} \|x_{n_i}-p\|^2+2\|q-p\|^2
$$
$$
=\lim_{n\rightarrow \infty} \|x_n-p\|^2+2\|q-p\|^2.
$$
Hence $p=q$ and, since $Fix\,(T_\lambda)=Fix\,(T)$, the conclusion of the theorem follows. 
\end{proof}

\begin{remark}
1) If $T$ is enriched $(0,k)$-strictly pseudocontractive, that is, $T$ is $k$-strictly pseudocontractive, then by Theorem \ref{th5} we obtain Theorem 3.1 in \cite{Mar}.

2) Our Theorem \ref{th5} extends the corresponding results in \cite{BroP67}, by considering  Krasnoselskij-Mann iteration instead of the simple Krasnoselskij iteration.

3) If $b=0$ and $k=0$ in \eqref{eq3}, then $T$ is nonexpansive and our Theorem \ref{th5} reduces to Reich's theorem \cite{Reich} in the Hilbert space setting. It is known, see \cite{Reich}, that if $T$ is nonexpansive and the control sequence $\{\alpha_n\}$ satisfies \eqref{eq7-1u}, then the sequence $\{x_n\}$ generated by Mann's algorithm
$$
x_{n+1}=(1-\alpha_n) x_n+\alpha_n  T x_n,\,n\geq 0,
$$
converges weakly to a fixed point of $T$ in  a uniformly convex Banach space with a Fr\' echet differentiable norm. 

It is therefore an open problem whether Reich's theorem can be extended to $k$-strict pseudo-contractions, see \cite{Mar}, or enriched $(b,k)$-strict pseudo-contractions.
\end{remark}

\section{Conclusions}

In this paper  we introduced and studied the class of {\it enriched strictly pseudocontractive mappings}  in the setting of a Hilbert space $H$. We have shown that any enriched nonexpansive mapping defined on a bounded, closed and convex subset $C$ of $H$ has  fixed points in $C$ and that in order to approximate a fixed point of an enriched nonexpansive mapping,  we can use the Krasnoselskij iteration, for which we have proven a strong convergence result (Theorem \ref{th1}) as well as a weak convergence theorem (Theorem \ref{th4}). We also established a weak convergence theorem for a Krasnohselskij-Mann iteration (Theorem \ref{th5}) that extends an important result due to Marino and Xu \cite{Mar} from strictly pseudocontractive mappings to enriched strictly pseudocontractive mappings.

Our results extend some convergence theorems in \cite{BroP67} from strictly pseudocontractive mappings to enriched strictly pseudocontractive mappings and thus include many other important related results from literature as particular cases, see \cite{Ber02}, \cite{Ber02a}, \cite{Ber07}, \cite{BroP66}, \cite{Chi09}, \cite{Kra55}, \cite{Mar}, \cite{Pet}, \cite{Schu91}  etc. For other related developments, see \cite{Ber19a}-\cite{Ber19e}, \cite{BerP19}, \cite{Pac19}, \cite{Pac19a}. 

We illustrated the richness of the new class of mappings  by means of Example \ref{ex1}, which shows that, alongside all nonexpansive mappings, strictly pseudocontractive mappings  are also included in the class of {\it enriched strictly pseudocontractive mappings}. 

Note also that, similarly to the case of nonexpansive mappings, any enriched strictly pseudocontractive mapping is continuous.

\vskip 0.5 cm {\it  Department of Mathematics and Computer Science

North University Center at Baia Mare

Technical University of Cluj-Napoca 

Victoriei 76, 430122 Baia Mare ROMANIA

E-mail: vberinde@cunbm.utcluj.ro}

\vskip 0.5 cm {\it Academy of Romanian Scientists  (www.aosr.ro)

E-mail: vasile.berinde@gmail.com}

\end{document}